\documentclass[a4paper,12pt]{amsart}
\usepackage{vmargin}
\usepackage{amsfonts}
\usepackage{amssymb}
\usepackage{mathrsfs}
\usepackage{stmaryrd}
\usepackage{amsmath}
\usepackage{amsthm}
\usepackage{enumerate}
\usepackage{nomencl}
\usepackage[bookmarks=true]{hyperref}   
\usepackage{esint}
\usepackage{dsfont}
\usepackage{upgreek}
\usepackage{color}
\usepackage{comment}
\usepackage{float}

\makeatletter
\renewcommand\section{\@startsection{section}{1}{\z@}%
                       {-3\p@ \@plus -4\p@ \@minus -4\p@}%
                       {3\p@ \@plus 4\p@ \@minus 4\p@}%
                      {\normalfont\normalsize\centering\scshape}}
\makeatother

\hypersetup{
    colorlinks,%
}

\author{Lashi Bandara and Ognjen Milatovic}
\title{Self-adjointness of the Gaffney Laplacian on vector bundles}
\date{\today}


\address{Lashi Bandara, Mathematical Sciences,
Chalmers University of Technology and University of Gothenburg, SE-412 96, Gothenburg, Sweden}
\urladdr{\href{http://www.math.chalmers.se/~lashitha}{http://www.math.chalmers.se/~lashitha}}
\email{\href{mailto:lashi.bandara@chalmers.se}{lashi.bandara@chalmers.se}}

\address{
Ognjen Milatovic,
Department of Mathematics and Statistics,
University of North Florida,
Jacksonville, FL  32224,
USA}
\urladdr{\href{http://www.unf.edu/~omilatov}{http://www.unf.edu/~omilatov}}
\email{\href{mailto:omilatov@unf.edu}{omilatov@unf.edu}}

\def\colour{\colour}

\def\colour{\color}

\newtheorem{theorem}{Theorem}[section]
\newtheorem{corollary}[theorem]{Corollary}
\newtheorem{lemma}[theorem]{Lemma}
\newtheorem{proposition}[theorem]{Proposition}

\newtheorem{remark}[theorem]{Remark}

\newtheorem{question}{Question}

\newcommand{\mdot}{\cdotp}

\newcommand{\cbrac}[1]{\left(#1\right)}

\newcommand{\dbrac}[1]{\left\{#1\right\}}
\newcommand{\modulus}[1]{|#1|}
\newcommand{\set}[1]{\dbrac{#1}}
\newcommand{\nset}[1]{\{#1\}}

\newcommand{\dom}{ {\mathcal{D}}}

\newcommand{\R}{\mathbb{R}}
\newcommand{\C}{\mathbb{C}}

\newcommand{\script}[1]{\mathscr{#1}}

\renewcommand{\emptyset}{\varnothing}

\newcommand{\intersect}{\cap}

\newcommand{\rest}[1]{{{\lvert_{}}_{}}_{#1}}
\newcommand{\close}[1]{\overline{#1}}		

\renewcommand{\epsilon}{\varepsilon}

\renewcommand{\phi}{\varphi}

\newcommand{\isomorphic}{\cong}			

\newcommand{\tensor}{\otimes}

\newcommand{\norm}[1]{\| #1 \|}			
\newcommand{\spt}[1]{{\rm spt} {\text{ }}#1}	

\DeclareMathOperator{\tr}{tr}			

\DeclareMathOperator{\divv}{div}		

\newcommand{\Sect}{\mathbf{\Gamma}}		
\newcommand{\tanb}{{\rm T}}		
\newcommand{\cotanb}{{\rm T}^\ast}	

\DeclareFontFamily{OT1}{restrictfont}{}
\DeclareFontShape{OT1}{restrictfont}{m}{n}{<-> fmvr8x}{}

\newcommand{\adj}[1]{{#1}^\ast}			
\newcommand{\extd}{{\rm d}}			

\newcommand{\inprod}[1]{\langle #1 \rangle}	
\newcommand{\grad}{\nabla}			
\newcommand{\conn}[1][{}]{{\grad_{{#1}}}}		

\newcommand{\Lp}[2][{}]{{\rm L}^{#2}_{\rm #1}}		
\newcommand{\Ck}[2][{}]{{\rm C}^{#2}_{\rm #1}}		
\newcommand{\Sob}[2][{}]{{\rm W}^{#2}_{\rm #1}}		

\DeclareMathOperator{\Cpc}{\textrm{Cap}}	

\newcommand{\Hil}{\script{H}}			
\newcommand{\Lap}{\Delta}			

\newcommand{\sD}{\script{D}}
\newcommand{\sE}{\script{E}}

\newcommand{\sR}{\script{R}}

\newcommand{\cV}{\mathcal{V}}
\newcommand{\cM}{\mathcal{M}}

\newcommand{\mg}{\mathrm{g}}

\newcommand{\mh}{\mathrm{h}}

\newcommand{\RNum}[1]{\uppercase\expandafter{\romannumeral #1\relax}}

\begin{document}

\maketitle

\begin{abstract}
We study the Gaffney Laplacian on
a vector bundle
equipped with a compatible 
metric and
connection over a
Riemannian manifold
that is possibly 
geodesically incomplete.
Under the hypothesis
that the Cauchy boundary is polar, 
we demonstrate the self-adjointness 
of this Laplacian.
Furthermore, we show that 
negligible boundary is a necessary and sufficient condition 
for the self-adjointness of this operator.
\end{abstract}
\vspace*{-0.5em}
\tableofcontents
\vspace*{-2em}

\parindent0cm
\setlength{\parskip}{\baselineskip}

\section{Introduction}

The study of the essential
self-adjointness of the Laplace-Beltrami operator $\Delta$ on a
geodesically complete Riemannian manifold $(\cM,\mg)$ in $\Lp{2}(\cM)$ and
the standard $\Lp{2}$-spaces of differential forms was
initiated by Gaffney in \cite{Gaffney}. Later, also in the setting of geodesically
complete Riemannian manifolds, Cordes in \cite{Cordes} proved the essential
self-adjointness of positive integer powers of the operator
$\Delta$ on $\Ck[c]{\infty}(\cM)$,  the space of  smooth compactly supported functions on $\cM$.
In contrast to Cordes' ``stationary" approach,
Chernoff in \cite{Chernoff} used a wave equation method to establish the essential
self-adjointness of positive integer powers of the Laplace operator on differential forms.

In \cite{Bandara-13}, the first author considered
the case of a vector bundle $\cV$
equipped with a metric $\mh$ and connection $\conn$ that are compatible,
over a geodesically complete manifold.
On factorising certain ``density problems''
in terms of a first-order operator
and applying the results of Chernoff from \cite{Chernoff},
he established the
density of $\Ck[c]{\infty}(\cV)$
in the Sobolev space $\Sob{1,2}(\cV)$ defined as the closure of
\begin{equation}\label{E:S-1}
S_1(\cV):=\{u \in \Ck{\infty} \intersect \Lp{2}(\cV)\colon \conn u \in \Ck{\infty} \intersect \Lp{2}(\cotanb\cM \tensor \cV)\}
\end{equation}
under the Sobolev norm $\norm{\mdot}_{\Sob{1,2}} = \norm{\mdot} + \norm{\conn \mdot}$, where $\|\cdot\|$ stands for the $\Lp{2}$-norm (see \S\ref{S:prelim} for details). 
Additionally, the author of~\cite{Bandara-13} established
the density of $\Ck[c]{\infty}(\cotanb\cM \tensor \cV)$
in the domain of
the divergence operator (the adjoint of $\conn$ in $\Lp{2}$)
as well as the essential self-adjointness
of the Bochner Laplacian $-\tr \conn^2$
on $\Ck[c]{\infty}(\cV)$,  the trace with respect
to the metric $\mg$ of the Hessian $\conn^2$ (c.f. Chapter 2 in \cite{Petersen} by Petersen).

In the context of Riemannian manifolds $(\cM, \mg)$ that are possibly
geodesically incomplete,  Masamune in \cite{Masamune} studied the essential
self-adjointness of the
Laplace-Beltrami operator
on functions. In particular, he showed that the essential
self-adjointness is equivalent to the \emph{negligible boundary property}
(see \S\ref{S:neg-bound}).
In the same paper, he showed that this
property is equivalent to the equality $\Sob[0]{1,2}(\cM)=\Sob{1,2}(\cM)$
(where the space $\Sob[0]{1,2}(\cM)$ is the closure of $\Ck[c]{\infty}(\cM)$
with respect to the Sobolev norm).

On a related note,
it turns out that the zero capacity of the Cauchy boundary of a
Riemannian manifold implies $\Sob[0]{1,2}(\cM)=\Sob{1,2}(\cM)$;
see \cite{Masamune-99, Masamune} by Masamune and
\cite{GM} by Grigor'yan and Masamune. A related study of the essential
self-adjointness of the sub-Laplacian on a sub-Riemannian manifold
with the pseudo-0-negligible boundary property can be found in
\cite{Masamune-05-CPDE} by Masamune.

In the present paper, we consider a
vector bundle $\cV$ with a compatible
metric $\mh$ and connection
$\conn$ over
a Riemannian manifold $(\cM,\mg)$
(without boundary) that is possibly geodesically incomplete.
First, by studying the effect of
zero capacity of the Cauchy boundary on $(\cV,\mh,\conn)$,
we show
that $\Sob[0]{1,2}(\cV) = \Sob{1,2}(\cV)$ (Theorem \ref{T:sob-equal}).

We then proceed to extend the notion of negligible boundary to $(\cM,\cV,\conn)$. Our fundamental object of study is the operator $\Delta_{G}:=\adj{\conn[c]}\,\close{\conn[2]}$, where $\adj{\conn[c]}$ is the operator adjoint of $\conn[c]:=\conn|_{\Ck[c]{\infty}(\cV)}$ and
$\close{\conn[2]}$ is the closure of $\conn[2] = \conn|_{S_1(\cV)}$ with  $S_1(\cV)$ as in~(\ref{E:S-1}); see \S\ref{S:prelim} for details. In the context of scalar-valued functions, $\Lap_G$ appears in \cite{GM}
by Grigor'yan and Masamune under the name \emph{Gaffney Laplacian}.
 We remark that, Gaffney actually did not study
this Laplacian, but to avoid inconsistencies with the literature,
we will retain this nomenclature in our more general setting. 
In Theorem~\ref{T:self-G}, we show that the self-adjointness of $\Delta_{G}$
is equivalent to the negligible boundary property, which is, in turn,
equivalent to the equality $\Sob[0]{1,2}(\cV)=\Sob{1,2}(\cV)$.
Additionally, in Theorem~\ref{T:esa-Lap-G-infty}, which is a
Bochner Laplacian analogue of Theorem 3 in \cite{Masamune},
we establish the equivalence between essential self-adjointness of
the Bochner Laplacian and negligible boundary.
 We emphasise that
our results in particular hold for the Laplacian on vectorfields,
co-vector fields, and more generally tensor fields, without requiring
underlying connection to be Levi-Civita. 

Our analysis
rests upon an application of a version of
integration by parts, which enables us to establish useful links among the first-order
operators entering the definitions of Dirichlet, Neumann, and
Gaffney Laplacians; see Theorem~\ref{T:op-equality} below. Other
useful links among those operators are obtained by extracting information
from the negligible boundary property. Additionally, in
the proof of the essential self-adjointness
of the Bochner Laplacian (Theorem~\ref{T:esa-Lap-G-infty}),
we adapt the heat equation method of
Masamune in \cite{Masamune,Masamune-05-CPDE} to our setting.
\section*{Acknowledgements}

The first author acknowledges the
support of the Australian National University
and the Australian Research Council
through Discovery Project DP120103692 of McIntosh and Portal,
as well as his current institution,
Chalmers University of Technology/University of Gothenburg, Sweden.
 The authors are grateful to the anonymous referee
for providing us with valuable suggestions and helping us improve the presentation of the material. 
\section{Preliminaries}\label{S:prelim}

Let $\cM$ be a smooth manifold
with a smooth metric $\mg$.
We emphasise that $\cM$
is a manifold \emph{without boundary}
and that $\mg$ is possibly \emph{geodesically incomplete}, 
 by which we mean that it may be incomplete as a metric space. 
We denote the induced volume measure by $\mu_\mg$.
We assume that $\cV$ is a smooth vector bundle
over $\cM$, equipped with a smooth metric
$\mh$.

Recall that a map 
$\conn: \Ck{\infty}(\cV) \to \Ck{\infty}(\cotanb\cM \tensor \cV)$
is called a \emph{connection} if
it satisfies $\conn[fX]v = f\conn[X]v$
and $\conn[X](fv) = \extd f \tensor v + f \conn[X]v$
where $f \in \Ck{\infty}(\cM)$, $X \in \Ck{\infty}(\tanb\cM)$, 
$v \in \Ck{\infty}(\cV)$ and $\extd f$ is the
exterior derivative of $f$.
A connection $\conn$ and metric $\mh$
are \emph{compatible} if $\conn \mh \equiv 0$.
We fix a connection $\conn$ on $\cM$
(by which we mean on $\tanb\cM$) and
by the same symbol, we denote a connection 
on $\cV$, compatible with $\mg$ and $\mh$
respectively. We do not assume that
$\conn$ on $\cM$ is \emph{Levi-Civita}; that
is, we allow for the existence of vectorfields
$X, Y \in \Ck{\infty}(\cotanb\cM)$
such that $\conn[X]Y \neq [X,Y]$
(where $[\mdot,\mdot]: \Ck{\infty}(\tanb\cM \times \tanb\cM) \to \Ck{\infty}(\tanb\cM)$
is the \emph{Lie bracket}). 

We define the Lebesgue spaces $\Lp{p}(\cV)$
for $p \geq 1$,
although we only require this
theory for the case $p = 2$. First, we note
some measure theoretic notions.
We say that a set $A \subset \cM$
is measurable if, whenever $(U, \psi)$
is a chart and $U \intersect A \neq \emptyset$,
then the set $\psi(A \intersect U)$ is Lebesgue measurable.
Thus, we define the space of measurable
sections $\Sect(\cV)$, where we
say that a section is measurable
if its coefficients are measurable
when seen through trivialisations.
We remark that
this notion of measurability is equivalent
to measurability with respect to the
induced measure $\mu_\mg$. See \S3 of \cite{B-Rough} for details.

For $p \in [1,\infty)$,
the set $\Lp{p}(\cM; \cV)$ (or $\Lp{p}(\cV)$ for short)
is defined as
the set of sections $\xi \in \Sect(\cV)$
such that
$$ \int_{\cM} \modulus{\xi(x)}_{\mh(x)}^p\ d\mu_\mg(x) < \infty.$$
We bestow $\Lp{p}(\cV)$ with the norm
$\norm{\xi}_{p}
=\cbrac{\int_{\cM} \modulus{\xi}_{\mh}^p\ d\mu_\mg}^{\frac{1}{p}}.$
In the case of $p = \infty$,
we define $\Lp{\infty}(\cV)$ to consist of sections
$\xi \in \Sect(\cV)$ such that there exists $C > 0$
with $\modulus{\xi(x)}_{\mh(x)} \leq C$  for
$x$-a.e. The norm $\norm{\xi}_\infty$ is
the infimum over all such constants $C$.
Each of these spaces is a Banach space
(strictly speaking, modulo
sections which differ on a set of measure zero).
The space $\Lp{2}(\cV)$ is a Hilbert space
with inner product
$$
\inprod{\xi,\zeta} = \int_{\cM} \mh_x(\xi(x), \zeta(x))\ d\mu_\mg(x)$$
for $\xi, \zeta \in \Lp{2}(\cV)$.
We assume these spaces are complex
valued by identifying a real space with
its complexification.
The \emph{local} $\Lp{p}$ spaces
are denoted by  $\Lp[loc]{p}(\cV)$
and they contain sections
$\xi \in \Sect(\cV)$ satisfying:
$\xi \in \Lp{p}(K, \cV)$
for every open $K \subset \cM$
with $\close{K}$ compact.

Let
$$S_k^p(\cV) = \set{u \in \Ck{\infty} \intersect \Lp{p}(\cV):
	\conn^i u \in \Ck{\infty} \intersect \Lp{p}((\tensor_{j=1}^i \cotanb\cM) \tensor \cV),\ i = 1, \dots, k},$$
and define
$\Sob{k,p}(\cV)$ as the closure of $S_k^p(\cV)$ under the norm
$$ \norm{u}_{\Sob{k,p}} = \norm{u}  + \sum_{i=1}^k \norm{ \conn^i u}.$$
Furthermore, let
$\Sob[0]{k,p}(\cV)$ denote the closure of $\Ck[c]{\infty}(\cV)$
under the same norm.
We define $\Sob[c]{k,p}(\cV)$ to be sections $\xi \in \Sob[0]{k,p}(\cV)$
with support $\spt \xi$ compact. The local Sobolev spaces
$\Sob[loc]{k,p}(\cV)$ consist of sections $\xi \in \Sect(\cV)$
such that $\xi \in \Sob{k,p}(K, \cV)$ for every
open $K \subset \cM$ with $\close{K}$ compact.
From here on, unless otherwise stated,
we will be solely concerned with the case $p = 2$.
Hence, we denote $S_k^2(\cV)$ by $S_k(\cV)$.

We also note the following characterisation:
\begin{equation*}\label{R:sob-eq}
\tag{W}
\Sob{1,2}(\cV) = \set{u \in \Lp{2}(\cV): \conn u \in \Lp{2}(\cotanb \cM \tensor\cV)}.
\end{equation*}

For this, we cite Theorem 2 (i) in \cite{Masamune},
which is a version of~(\ref{R:sob-eq}) for functions.
The proof generalises with only
superficial modifications to the vector bundle setting as it relies purely
upon the properties of Friedrichs mollification; see the proof of Proposition 1.4 in~\cite{E-88}. 

Define $\conn[c]: \Ck[c]{\infty}(\cV) \to \Ck[c]{\infty}(\cotanb\cM \tensor \cV)$
and $\conn[2]: S_1(\cV) \to \Ck{\infty}(\cotanb\cM \tensor \cV) \intersect \Lp{2}(\cotanb\cM \tensor \cV)$
by $\conn[c] = \conn$ with domain  $\dom(\conn[c]) = \Ck[c]{\infty}(\cV)$ and
$\conn[2] = \conn$ with domain $\dom(\conn[2]) = S_1(\cV)$.

In particular, the compatibility of $\conn$
with $\mh$ induces
the following adjoint formulae:
\begin{alignat*}{4}
&\inprod{u, -\tr\conn[c] v}
&&= \inprod{\conn[2]u, v},
&&\quad\text{where}\ u \in S_1(\cV),\ v \in \Ck[c]{\infty}(\cotanb\cM\tensor \cV),
\tag{N}
\label{Eq:AdjointN}
\\
&\inprod{w, -\tr\conn[2] z}
&&= \inprod{\conn[c]w, z},
&&\quad\text{where}\ w \in \Ck[c]{\infty}(\cV),\ z \in S_1(\cotanb\cM\tensor \cV).
\tag{D}
\label{Eq:AdjointD}
\end{alignat*}
 We emphasise that the compatibility of $\conn$ with $\mh$ is
essential as the results we obtain crucially 
depend on these adjoint formulae. 

Since $\Ck[c]{\infty}(\cV)$ is dense in $\Lp{2}(\cV)$,
we obtain by standard theory (i.e. Theorem III.5.28 in \cite{Kato} by Kato)
that the operators
$\conn[c]$, $\conn[2]$, $-\tr \conn[c]$ and $-\tr \conn[2]$
are densely-defined and closable.
Hence, we define the following operators:
\begin{align*}
&\conn[D] := \close{\conn[c]},\quad  &&\conn[N] := \close{\conn[2]}, \\
&\divv_{D} := \adj{-\conn[c]},\quad &&\divv_{N} := \adj{-\conn[2]}, \\
\ &\divv_{D,-} := \close{\tr\conn[2]},\quad &&\divv_{N,-} := \close{\tr\conn[c]}.
\end{align*}

First, observe the following.
\begin{proposition}\label{P:basic-incl}
The following operator inclusions hold:
$\conn[D] \subset \conn[N]$,
$\divv_N \subset \divv_D$,
$\divv_{D,-} \subset \divv_D$, $\divv_{N,-} \subset \divv_N$.
\end{proposition}

Moreover, we obtain a characterisation
of the Sobolev spaces in terms of
$\conn_D$ and $\conn_N$.

\begin{proposition}\label{P:basic-eql}
The Sobolev space
$\Sob{1,2}(\cV) = \dom(\conn[N])$ and
the Sobolev space
$\Sob[0]{1,2}(\cV) = \dom(\conn[D]).$
\end{proposition}

We remark that the operators $\divv_D$ and $\divv_N$ can be obtained
as closed, densely-defined operators even if the compatibility
assumption on $\mh$ and $\conn$ is dropped.
This is a consequence of the well known
fact that operators $\conn_c$ and $\conn_2$
are always densely-defined and closable (c.f.~Proposition 2.2 in \cite{BMc}).
In particular, this means that Proposition \ref{P:basic-eql}
is valid even in this more general context.
However, the inclusions $\tr\conn_c \subset \divv_N$
and $\tr\conn_2 \subset \divv_D$ may no longer hold.
In particular, we cannot
assert that $\divv_D$ and $\divv_N$ are differential
operators.

Dropping the compatibility requirement
becomes crucial when attempting to study
Sobolev spaces in the setting of low-regularity
metrics, for the simple reason that
the metric may not be differentiable.
Some initial progress
in this direction can be found in \cite{B-Rough}
for the special case of Sobolev spaces
of functions under so-called ``rough metrics.''
These considerations are beyond the scope
of this paper and we will always assume
compatibility between the metric and connection
unless otherwise stated.

Define the following two self-adjoint
operators, which we respectively call the \emph{Dirichlet} and
\emph{Neumann} Laplacian:
\begin{align*}
\Lap_D := -\divv_D \conn[D]\ \text{and}\ \Lap_N := -\divv_N \conn[N].
\end{align*}
On writing the \emph{energy} associated to our Sobolev spaces,
namely,
$$ E_D[u] = \int_{\cM} \modulus{\conn[D]u}^2\ d\mu_\mg
\ \text{and}\
E_N[v] = \int_{\cM} \modulus{\conn[N]v}^2\ d\mu_\mg,$$
for $u \in \Sob[0]{1,2}(\cV)$ and $v \in \Sob{1,2}(\cV)$,
it is immediate that
$\Sob[0]{1,2}(\cV) = \dom(\sqrt{\Lap_D})$
and $\Sob{1,2}(\cV) = \dom(\sqrt{\Lap_N})$.
Note that when $\cM = U \subset \R^n$,
where $U$ is the interior of a bounded
Lipschitz domain and $\cV = \cM \times \C$, the bundle of functions,
then $\Lap_D$ and $\Lap_N$ denote
the classical Dirichlet and Neumann Laplacians respectively.
This justifies our notation and nomenclature
in this more general setting.

We also consider the composition of
operators $-\tr\conn[2]$ and $\conn[2]$:
\[
\Lap_{s}:=(-\tr\conn[2])\conn[2],
\]
with the induced domain
$\dom(\Lap_s) = \set{u \in S_1(\cV): \conn u \in S_1(\cotanb\cM \tensor \cV)}$.
We call this operator the \emph{Bochner Laplacian}.
Let $\Lap_{D,s}$ and $\Lap_{N,s}$ be the restrictions of $\Lap_{D}$ and $\Lap_{N}$ to $\dom(\Lap_{D})\cap \Ck{\infty}(\cV)$ and
$\dom(\Lap_{N})\cap \Ck{\infty}(\cV)$ respectively. The following operator relations easily follow from definitions and Proposition~\ref{P:basic-incl}:
\begin{equation*}\label{E:D-N-s}
\tag{L}
\Lap_{D,s}\subset \Lap_{s}\quad\textrm{ and }\quad\Lap_{N,s}\subset \Lap_{s}.
\end{equation*}

Furthermore, a fundamental object
of our study will be the following (not necessarily symmetric) operator:
\begin{align*}
\Lap_G := -\divv_D \conn[N].
\end{align*}
 
As aforementioned, we call this object the Gaffney Laplacian
to keep consistency with the literature. 
\section{General results}\label{S:gen-results}

Let us consider an additional
$\Lp{2}$-differential operator $\conn[L] := \adj{(-\divv_{N,-})}$.
Recall the containment
$\divv_{N,-} \subset \divv_N \subset \divv_D$
from Proposition \ref{P:basic-incl}.
On taking adjoints
we obtain that
$ \conn[D] \subset \conn[N] \subset \conn[L].$
We will show that $\conn[L] = \conn[N]$,
but first, we present the
following abstract lemma
inspired by this endeavour.

\begin{lemma}
\label{Lem:AdjDom}
Let $\Hil_1$ and $\Hil_2$ be Hilbert spaces,
with inner products $\inprod{\mdot,\mdot}_1$
and $\inprod{\mdot,\mdot}_2$ respectively.
Let $\sR_s \subset \Hil_2$ and  $\sR \subset \Hil_2$, be dense subspaces
satisfying $\sR_s \subset \sR$.
Suppose that:
\begin{enumerate}[(i)]
\item $\sR$ is equipped with a norm $\norm{\mdot}_{\sR}$
	satisfying: there exists $c > 0$ such that
	 $\norm{u}_2 \leq c \norm{u}_{\sR}$
	for all $u \in \sR$,
\item the inner product
	$\inprod{\mdot,\mdot}_2$ extends
	continuously from $\sR_s$ to the
	pairing (separably continuous bilinear map)
	$$\inprod{\mdot,\mdot}: \sR \times \sR' \to \C,$$
	where $\sR'$ is the dual space of $\sR$,
\item  $T: \dom(T) \subset \Hil_1 \to \Hil_2$ and $\tilde{T}: \Hil_1 \to \sR'$ 
	and $\tilde{T}\rest{\dom(T)} = T$, 
\item  $S: \dom(S) \subset \Hil_2 \to \Hil_1$, $\sR_s \subset \dom(S) \subset \sR$,
	and $\dom(S)$ is dense in $(\Hil_2, \norm{\mdot}_2)$
	as well as $(\sR, \norm{\mdot}_{\sR})$, and 
\item $\inprod{u, Sv}_1 = \inprod{\tilde{T}u,v}$
	for all $u \in \Hil_1$ and $v \in \dom(S)$.
\end{enumerate}
Then, $\dom(\adj{S}) = \nset{u \in \Hil_1: \tilde{T}u \in \Hil_2}$.
\end{lemma}
\begin{proof}

Let $\dom = \nset{u \in \Hil_1: \tilde{T}u \in \Hil_2}$.
Combining (iii) and (v), it is easy to see that that $S$ and $T$
are adjoint to each other and therefore, $\dom \subset \dom(\adj{S})$.

To prove the converse, first
recall that $\dom(\adj{S})$ can be characterised
as the set of $u \in \Hil_1$ for which
$v \mapsto \inprod{u, Sv}_1$ is continuous.
Here, continuity is measured
with respect to the
topology induced by $\norm{\mdot}_2$ on $\Hil_2$.

Fix $u \in \dom(\adj{S})$. Then,
$\inprod{u, Sv}_1 = \inprod{\tilde{T}u,v}$
by (v) and $v \mapsto \inprod{\tilde{T}u,v}$ is
continuous in $\Hil_2$.
Letting $f_u$ be this map,
continuity is equivalent to
 $\modulus{f_u(v)} = \modulus{\inprod{\tilde{T}u,v}} \leq C \norm{v}_{2}$
for some $C > 0$ whenever $v \in \dom(S)$.
Since $\dom(S)$ is dense in $\Hil_2$ by (iv), we have that $f_u(v)$ extends
to the whole of $\Hil_2$ as a bounded map. That is,
$f_u \in \Hil_2'$ and hence,
by the Riesz representation theorem,
there exists $z \in \Hil_2$ such that
$f_u(v) = \inprod{z,v}_2$ for all $v \in \Hil_2$.
Now, note $\modulus{f_u(v)} \leq C \norm{v}_2 \leq cC \norm{v}_{\sR}$
where the second inequality follows form (i)
and hence, $\Hil_2 \isomorphic \Hil_2' \subset \sR'$. That is, $z \in \sR'$. 
Since $\inprod{\mdot,\mdot}$ extends $\inprod{\mdot,\mdot}_2$ from $\sR_s$
via continuity,
we obtain that $\inprod{\tilde{T}u,v} = \inprod{z,v}_2 = \inprod{z, v}$
for all $v \in \dom(S)$ by invoking (ii). By the density of $\dom(S)$
in $\sR$ guaranteed by (iv), we can deduce that $\tilde{T}u = z$
and $\tilde{T}u \in \Hil_2$ since $z \in \Hil_2$. This proves $u \in \dom$. 
\end{proof}

We remark that the formulation of this
lemma is to provide a tool to compute
the domain of an operator when there
are distributional tools
in hand. In application,
we will see that $\sR'$ represents
the space of distributional sections,
$\sR$ an appropriate Sobolev space,
and $\sR_s$, the space of compactly
supported smooth sections.
We have phrased this lemma in this generality
in the hope that it will be useful
beyond the scope of our immediate
applications in this paper.

Let $\sE'(\cV)$ be the dual space of
$\Ck{\infty}(\cV)$ and $\sD'(\cV)$ the dual
space of $\Ck[c]{\infty}(\cV)$.
Define
$\Sob[comp]{1,2}(\cV) := \Sob[loc]{1,2}(\cV) \intersect \sE'(\cV)$.
The choice of the notation ``comp'' in the definition
is because $\Sob[comp]{1,2}(\cV)$ is
exactly the $\Sob{1,2}(\cV)$
sections with compact support, 
which is a consequence of the fact that
 $\sE'(\cV)$ is the
space of compactly supported distributional sections; see Exercise 2.3.5 in~\cite{BC-2009} by van den Ban and Crainic.

\begin{lemma}
\label{Lem:Dual}
The $\Lp{2}(\cV)$ inner product $\inprod{\mdot,\mdot}$
extends continuously to a pairing
$$
\inprod{\mdot,\mdot}:\Sob[comp]{1,2}(\cV) \times \Sob[loc]{-1,2}(\cV) \to \C$$
from $\Ck[c]{\infty}(\cV)$ by continuity.
Furthermore, the equality
\begin{equation*}\label{E:i-parts}
\tag{P}
\inprod{\conn u, v} = \inprod{u, -\tr \conn v}
\end{equation*}
holds if
one of $\spt u$  or $\spt v$ is compact, and
either
\begin{enumerate}[(i)]
\item $u \in \Lp[loc]{2}(\cV)$ and $v \in \Sob[loc]{1,2}(\cotanb \cM \tensor \cV)$, or
\item $u \in \Sob[loc]{1,2}(\cV)$ and $v \in \Lp[loc]{2}(\cotanb \cM \tensor \cV)$.
\end{enumerate}
\end{lemma}

For the statement about the pairing in Lemma~\ref{Lem:Dual}, see Lemma 9.2.9 in~\cite{BC-2009}.
For the equality~(\ref{E:i-parts}), see Lemma 8.8 in
the paper~\cite{BMS} by Braverman, the
second author and Shubin. A good reference for
similar results is \S7 (particularly
\S7.7) and Theorem 7.7 in \cite{Shubin} by Shubin.

\begin{theorem}\label{T:op-equality}
The following operator equalities hold:
\begin{enumerate}
\item[(i)] $\conn[L] = \conn[N]$
\item[(ii)] $\divv_{N}=\divv_{N,-}$
\item [(iii)]$\divv_{D}=\divv_{D,-}$
\end{enumerate}
\end{theorem}
\begin{proof}
Recall that $\conn[L] = \adj{(-\close{\tr\conn[c]})} = \adj{(-\tr\conn[c])}$.
First, we use Lemma \ref{Lem:Dual} and Lemma \ref{Lem:AdjDom}
and show that $\dom(\conn[L]) = \set{u \in \Lp{2}(\cV): \conn u \in \Lp{2}(\cotanb \cM \tensor \cV)}$.
To that end, let
$\Hil_1 = \Lp{2}(\cV)$, $\Hil_2 = \Lp{2}(\cotanb\cM \tensor \cV)$,
$S = -\tr\conn[c]$ with $\sR_s = \dom(S) = \Ck[c]{\infty}(\cotanb\cM \tensor \cV)$
and $T = \conn[2]$.
Now, let $\sR = \Sob[comp]{1,2}(\cotanb\cM \tensor \cV)$ 
 with norm $\norm{u}_{\sR} = \norm{\conn u} + \norm{u}$ 
so that $\sR' = \Sob[loc]{-1,2}(\cotanb\cM \tensor \cV)$
since $\Sob[loc]{-1,2}$ is the dual of $\Sob[comp]{1,2}$ (see Lemma 9.2.9 in \cite{BC-2009}).
 On noting that $\conn:\Lp{2}(\cV) \to \Sob[loc]{-1,2}(\cotanb\cM \tensor \cV)$,
we define $\tilde{T} = \conn$. 
Also, $\dom(S) = \Ck[c]{\infty}(\cotanb\cM \tensor\cV) \subset \Sob[comp]{1,2}(\cotanb\cM \tensor \cV) = \sR$
and
$\sR = \Sob[comp]{1,2}(\cotanb\cM \tensor \cV) \subset \Lp{2}(\cotanb\cM \tensor \cV) = \Hil_2$.
Thus, we have shown  that the hypotheses (i), (iii), (iv) of Lemma \ref{Lem:AdjDom} are satisfied.
Furthermore, hypothesis (ii) of Lemma \ref{Lem:AdjDom} is satisfied by Lemma \ref{Lem:Dual}.
Finally, the fulfilment of hypothesis (v) of Lemma \ref{Lem:AdjDom} follows from~(\ref{E:i-parts}) with $u\in \Lp{2}(\cV)$ and $v\in \Ck[c]{\infty}(\cotanb\cM \tensor\cV)$.
Thus, we conclude that $\dom(\conn[L]) = \set{u \in \Lp{2}(\cV): \conn u \in \Lp{2}(\cotanb \cM \tensor \cV)}$, and the equality $\dom(\conn[L])= \Sob{1,2}(\cV)$ follows from~(\ref{R:sob-eq}). This proves property (i).

Property (i) implies $\adj{\conn[L]} = \adj{\conn[N]}$, which immediately gives property (ii).
For property (iii), we use the $\Lp{2}$-space notations $\Hil_1=\Lp{2}(\cotanb\cM \tensor \cV)$, $\Hil_2 = \Lp{2}(\cV)$, operators $S= -\conn[c]$ and $T=\tr\conn[2]$, Sobolev space $\sR = \Sob[comp]{1,2}(\cV)$, and the set $\sR_s = \dom(S) = \Ck[c]{\infty}(\cV)$. With these notations,  property (iii) follows from Lemma~\ref{Lem:AdjDom} and Lemma~\ref{Lem:Dual} by using the same arguments as in the proof of property (i).
\end{proof}

The following proposition is a vector bundle analogue of Lemma 3 from~\cite{Masamune}.
\begin{proposition}\label{L:closures-N-D}
The following equalities  hold:
\[
\close{\Lap_{D,s}}=\Lap_{D}\quad\textrm{ and }\quad\close{\Lap_{N,s}}=\Lap_{N}.
\]
\end{proposition}
\begin{proof} As the same kind of proof applies to both equalities, we will only prove the first one.
We recall~(\ref{E:D-N-s}) and apply the heat-equation method of Masamune.
Since $-\Lap_{D}$ is a non-positive self-adjoint operator, it generates a strongly continuous
contraction semigroup $(e^{-t\Lap_{D}})_{t\geq 0}$ on $\Lp{2}(\cV)$.
Let $u\in\dom (\Lap_{D})$ be arbitrary, and consider the family $u_{t}:=e^{-t\Lap_{D}}u$.
By computing $u_t$ via the functional calculus for sectorial operators,
we can easily see that $e^{-t\Lap_D} \in\dom (\Lap_{D})$
(see \S{D} in \cite{ADMc}). 
Furthermore, since $\Lap_{D}$ is an elliptic operator, using elliptic regularity (see Corollary 7.1(b) and Corollary 7.4 in~\cite{Shubin}) we have $u_{t}\in \Ck{\infty}(\cV)$, and, therefore, $u_{t}\in \dom(\Lap_{D,s})$. Moreover, we have 
$$
u_t\to u\quad\textrm{as}\quad t\to 0^+,
$$
and since $e^{-t\Lap_{D}}$ commutes with $\Lap_{D}$
(the functional calculus commutes with the operator on its domain), we have
\[
\Lap_{D,s}u_t=\Lap_{D}u_t=e^{-t\Lap_{D}}\Lap_{D}u\to \Lap_{D}u, \quad \textrm{as }t\to 0+,
\]
where both convergence relations are understood in the $\Lp{2}$-sense. This proves the equality $\close{\Lap_{D,s}}=\Lap_{D}$.
\end{proof}

The next proposition is a vector bundle analogue of Lemma 3.6 (i) and (ii) from \cite{GM}.
\begin{proposition}\label{Prop:lapgdn} The following are equivalent:
\begin{enumerate}[(i)]
\item $\Sob[0]{1,2}(\cV)=\Sob{1,2}(\cV)$,
\item $\Lap_{D}=\Lap_{N}$,
\item $\Lap_{G}$ is self-adjoint.
\end{enumerate}
\end{proposition}
\begin{proof}
The equivalence of (i) and (ii) is easy. 
If $\Lap_{D}=\Lap_{N}$, then $\conn[N]=\conn[D]$. Hence, $-\divv_{D}=\conn[N]^*$, and, therefore, $\Lap_{G}=\Lap_{N}$. Thus, $\Lap_{G}$ is self-adjoint. As for the remaining implication, note that $\Lap_{D}\subset \Lap_{G}$ and $\Lap_{N}\subset \Lap_{G}$. Taking adjoints and using the self-adjointness of $\Lap_{D}$, $\Lap_{N}$, and $\Lap_{G}$, we get the equality (i).
\end{proof}

\begin{remark}
Note that this proposition holds
even if metric compatibility between
$\conn$ and $\mh$ is dropped,
upon defining $\divv_D$ and $\divv_N$
abstractly as adjoints of $-\conn_D$ and $-\conn_N$
respectively. 
\end{remark}
\section{Polar Boundary}
Let $\overline{\cM}$ be the metric completion of $\cM$ with respect to the Riemannian distance. We define the Cauchy boundary of $\cM$ as
\[
\partial_{C}\cM:=\overline{\cM}\backslash \cM.
\]

Following~\cite{GM}, we now give the definition of \emph{1-capacity} on $\cM$. Let $\mathcal{O}$ be the collection of all open sets of $\overline{\cM}$. For a set $\Omega\in \mathcal{O}$, we define
\[
\mathcal{L}(\Omega):=\{u\in \Sob{1,2}(\cM)\colon 0\leq u\leq 1\textrm{ and }u|_{\Omega\cap \cM}=1\}.
\]
We define the 1-capacity of $\Omega\in \mathcal{O}$ as follows:
\[
\Cpc(\Omega):=\inf_{u\in \mathcal{L}(\Omega)}\int_{\cM}(|u|^2+|\extd u|^2)\,d\mu_\mg \quad\textrm{if}\quad \mathcal{L}(\Omega)\neq\emptyset.
\]
Furthermore, define $\Cpc(\Omega)=\infty$ if $\mathcal{L}(\Omega)=\emptyset$ and $\Cpc(\emptyset)=0$. For an arbitrary set $\Sigma \subset\cM$, define
\[
\Cpc(\Sigma):=\inf_{\Omega\in\mathcal{O},\,\Sigma\subset\Omega}\Cpc(\Omega).
\]
We call $\Sigma$ \emph{polar} if $\Cpc(\Sigma)=0$. If $\Sigma=\emptyset$, we set $\Cpc(\Sigma)=0$.

We prove that polarity of the Cauchy boundary is
a sufficient condition to establish $\Sob[0]{1,2}(\cV) = \Sob{1,2}(\cV)$.
But first, we prove the following approximation lemma which
is noteworthy in its own right.

\begin{lemma}
\label{Lem:LinfDen}
The set $\Lp{\infty} \intersect \Sob{1,2}(\cV)$ is dense in $\Sob{1,2}(\cV)$.
\end{lemma}
\begin{proof}
We use the truncation procedure of Lemma 2 in \cite{Le-Si}
by Leinfelder and Simader.
By definition of $\Sob{1,2}(\cV)$, it is enough to show that $\Lp{\infty}(\cV)\intersect \Sob{1,2}(\cV)$ is dense in $S_1(\cV)$ with respect to $\Sob{1,2}$-norm. By Lemma 1.16 in~\cite{E-88} by Eichhorn, for all $v\in \Ck{\infty}(\cV)$ we have
the following diamagnetic inequality
\begin{equation*}\label{E:diam}
\tag{$\dagger$}
|\extd|v|_{\mh}|\leq |\conn v|,\qquad\mu_\mg\textrm{-a.e.}~x\in M,
\end{equation*}
where $|\cdot|_{\mh}$ is the norm with respect to the metric $\mh$ of $\cV$.
In particular,~(\ref{E:diam}) holds for all $u\in S_1(\cV)$.  Now using~(\ref{E:diam}) and~(\ref{R:sob-eq}) together, we conclude that $|u|_{\mh}\in \Sob{1,2}(\cM)$.
For simplicity we will suppress $\mh$ in $|u|_{\mh}$ for the remainder of the proof.

For $R>0$ define the following family of Lipschitz functions:
\[
\psi_R(t)=\left\{\begin{array}{cc}
    1, \qquad &\textrm{if}\quad t\leq R;\\
    \frac{R}{t}, \qquad &\textrm{if}\quad t > R.
    \end{array}\right.
\]
Note that $\psi_R'(t)=0$ if $t<R$, $0\leq \psi_R(t)\leq 1$, $|t\psi_R(t)|\leq R$ and $|t\psi_R'(t)|\leq 1$.

We now apply Theorem A in \cite{Le-Si} to the composition $\psi_{R}\circ |u|$.
We remark that Theorem A in \cite{Le-Si} was proven for the composition
$f\circ w$,  where $f\colon \mathbb{R}^{k}\to \mathbb{R}$ is a Lipschitz
function of class $C^1(\mathbb{R}^{k}\backslash\Gamma)$ with a
closed countable set $\Gamma\subset\mathbb{R}^{k}$,
and where $w\colon\Omega\to\mathbb{R}^{k}$ belongs to the
Sobolev space $\Sob{1,2}(\Omega,\C)$ with $\Omega$ an open set
in $\mathbb{R}^{m}$. However, the corresponding arguments can be carried over
without change to the case of functions $w$ defined on a Riemannian manifold.
Hence, we obtain
\begin{equation}\label{E:eqn-1}
\extd(\psi_{R}\circ |u|)=\psi_{R}'(|u|)\extd|u|\,\qquad\mu_\mg\textrm{-a.e.}~x\in M.
\end{equation}
As in the proof of Lemma 2 in \cite{Le-Si} we set
\[
u_{R}:=(\psi_{R}\circ |u|)u.
\]
Clearly, $u_R \in\Lp{\infty}(\cV)\cap\Sob{1,2}(\cV)$. Using Leibniz rule and~(\ref{E:eqn-1}),  we have
\[
\nabla u_{R}=(\nabla u)\psi_{R}(|u|)+u\otimes(\psi_{R}'(|u|)\extd|u|),
\]
in the sense of distributional sections of $\cV$.

Let $\chi_{G}$ denote the characteristic function of a set $G$. From the properties of $\psi_{R}$  it follows that
\[
|\nabla u_{R}-\nabla u|\leq (|\nabla u|+|\extd|u||)\chi_{\{|u|\geq R\}}
\]
and
\[
|u_{R}-u|\leq |u|\chi_{\{|u|\geq R\}}.
\]
Therefore, as $R\to\infty$, we have
$\|u_{R}-u\|_{\Sob{1,2}}\to 0.$
\end{proof}

\begin{remark} 
We note that the diamagnetic inequality
\eqref{E:diam} as proved by
Eichhorn in \cite{E-88} assumes
the compatibility of $\mh$ and $\conn$.
It would be interesting to know whether
this inequality still holds
without this assumption.
\end{remark}

The following is then a vector-bundle analogue of Lemma 2.2 (a) in~\cite{GM}.

\begin{theorem}\label{T:sob-equal} If $\partial_{C}M$ is polar, then $\Sob[0]{1,2}(\cV)=\Sob{1,2}(\cV)$.
\end{theorem}
\begin{proof}
The proof mimics that of Lemma 2.2(a) in~\cite{GM}.
For an arbitrary $u\in \Sob{1,2}(\cV)$ we construct a sequence $u_{j}\in \Sob[0]{1,2}(\cV)$ such that $\|u_j-u\|_{\Sob{1,2}}\to 0$ as $j\to\infty$.
By Lemma \ref{Lem:LinfDen}, we assume that $u\in\Lp{\infty}(\cV)\cap\Sob{1,2}(\cV)$.

Since $\Cpc(\partial_{C}M)=0$ there exists a sequence of open sets $\Omega_{k}\subset \overline{M}$, $k\geq 1$, such that $\partial_{C}M\subset  \Omega_{k+1}\subset \Omega_{k}$ and $\Cpc(\Omega_k)\to 0$ as $k\to\infty$. For each $k\geq 1$, let $\{\varphi_{j}^{(k)}\}_{j\geq 1}$ be a sequence of functions with the following properties:  $\varphi_{j}^{(k)}\in\mathcal{L}(\Omega_{k})$ and $\|\varphi_{j}^{(k)}\|_{\Sob{1,2}}\to \Cpc(\Omega_k)$, as $j\to\infty$. Define
\[
\varphi_{j}:=\varphi_{j}^{(j)}\quad\textrm{and}\quad  u_j:=(1-\varphi_{j})u.
\]
Fix a point $x_0\in \cM$ and define for
$r > 0$:

$$
\sigma_r(x):=\min\set{1,\ (r^{-1}(2r-d(x,x_0)))_{+}}, 
$$

where $d(\cdot,\cdot)$ is the distance with respect to the metric on $\cM$ and $f_{+}$ is the positive part of a function $f$.

Clearly, $\sigma_{r}\in \Sob{1,2}(\cM)$ for all $r>0$, $\sigma_{r}\to 1$, and $\extd\sigma_r\to 0$ as $r\to\infty$ in the $\mu_\mg$-a.e. sense. We define $v_{r,j}:=\sigma_r u_j$ and observe that $v_{r,j}\in \Sob[0]{1,2}(\cV)$ and $\|v_{r,j}-u_j\|_{{\Sob{1,2}}}\to 0$ as $r\to\infty$. Thus, we may assume (without loss of generality) that $u_j\in \Sob[0]{1,2}(\cV)$.

Since $(1-\varphi_{j})\to 1$ $\mu_\mg$-a.e. as $j\to\infty$, it follows that
\[
\|u_j-u\|_{2}\to 0, \quad\textrm{as }j\to\infty.
\]
Finally, using the properties $\varphi_{j}\to 0$ $\mu_\mg$-a.e., $\conn u\in \Lp{2}(\cotanb\cM \tensor \cV)$, $u\in \Lp{\infty}(\cV)$ and $\|\extd\varphi_j\|_{2}\to 0$, we obtain
\[
\|\conn u_j\|_{2}=\|(1-\varphi_{j})\conn u-\extd\varphi_{j}\otimes u\|_{2}\to \|\conn u\|_{2},
\]
as $j\to \infty$.
\end{proof}

The following corollary then follows directly from
Proposition~\ref{Prop:lapgdn} and Theorem~\ref{T:sob-equal}.
\begin{corollary} If $\partial_{C}M$ is polar, then $\Lap_{G}$ is self-adjoint.
\end{corollary}

We remark on our results here that,
much in the same way that the assumption of completeness 
yields the essential self-adjointness  
of the Bochner Laplacian (or equivalently
the self-adjointness of the Gaffney Laplacian) in \cite{Bandara-13}, the
polarity of Cauchy boundary is very much at the level of the length
structure of the manifold, and it is a sufficiently strong condition to
assert the self-adjointness of the Gaffney Laplacian.  
\section{Negligible Boundary}\label{S:neg-bound}

Let $(\cM,\cV,\conn)$ be as in \S\ref{S:prelim} and
recall the set
$$S_1(\cV) =  \set{u \in \Ck{\infty} \intersect \Lp{2}(\cV):
	\conn u \in \Ck{\infty}\intersect \Lp{2}(\cotanb\cM\tensor\cV)}.$$

We say that $(\cM,\cV,\conn)$ has \emph{negligible boundary} if
\begin{equation*}\label{E:n-b}
\tag{NB}
\inprod{\conn[2] u, v} = \inprod{u, -\tr\conn[2] v},
\quad \text{for all}\ u \in S_1(\cV),\ v \in S_1(\cotanb\cM\tensor \cV).
\end{equation*}

This definition is analogous to the one used by Masamune
in \cite{Masamune-99, Masamune}
in the study of the self-adjointness of the Laplacian acting on functions.
The term ``negligible boundary" goes back to
the work of Gaffney in \cite{Gaffney}.
In this section, we will illustrate the link between this
geometric condition \eqref{E:n-b}
and the equality $\Sob[0]{1,2}(\cV) = \Sob{1,2}(\cV)$.

First, we consider the relationship of $\eqref{E:n-b}$
to operators that we have introduced previously.
The following is the link between
negligible boundary and the Gaffney Laplacian.

\begin{theorem}\label{T:self-G} The operator $\Lap_{G}$ is self-adjoint if and only if $(\cM,\cV,\conn)$ has negligible boundary.
\end{theorem}
\begin{proof}  If $\Lap_{G}$ is self-adjoint, by Proposition~\ref{Prop:lapgdn} we get $\Sob[0]{1,2}(\cV)=\Sob{1,2}(\cV)$. Now Proposition~\ref{P:basic-eql} implies $\conn[D]=\conn[N]$. To verify the property~(\ref{E:n-b}), we first approximate $u\in S_1(\cV)$ by a sequence $u_j\in \Ck[c]{\infty}(\cV)$ in $W^{1,2}$-norm. Next, 
we recall the property ~(\ref{Eq:AdjointD})  which is a consequence
of the compatibility of $\conn$ and $\mh$:
\[
\inprod{\conn[c] u_j, v} = \inprod{u_j, -\tr\conn[2] v},
\quad \textrm{for all } v \in S_1(\cotanb\cM\tensor \cV).
\]
Finally, we take the limit as $j\to\infty$ on both sides, and this shows~(\ref{E:n-b}).

Now assume that $(\cM,\cV,\conn)$ has negligible boundary. Our goal is to show that $-\divv_{D}=\conn[2]^*$. From~(\ref{E:n-b}) it follows that $-\tr\conn[2]\subset \conn[2]^{*}$, which, after taking closures, leads to  $-\close{\tr\conn[2]}\subset \conn[2]^{*}$. By Theorem~\ref{T:op-equality} (iii), we can rewrite the last inclusion as $-\divv_{D}\subset \conn[2]^{*}$. Additionally, by Proposition~\ref{P:basic-incl} we have $-\divv_{N}\subset-\divv_{D}$, that is, $\conn[2]^{*}\subset-\divv_{D}$. Thus, we have shown that $-\divv_{D}=\conn[2]^*$. Noting that $\conn[2]^*=\conn[N]^*$, we have
\[
\Lap_{G}=\Lap_G = -\divv_D \conn[N]=\conn[2]^{*}\conn[N]=\conn[N]^*\conn[N].
\]
Now the self-adjointness of $\Lap_{G}$ follows by von Neumann's Theorem; see Theorem V.3.24 in~\cite{Kato}.
\end{proof}

The following theorem then links \eqref{E:n-b}
to the Bochner Laplacian $\Lap_{s}$.

\begin{theorem}\label{T:esa-Lap-G-infty} The operator $\Lap_{s}$ is essentially self-adjoint if and only if $(\cM,\cV,\conn)$ has negligible boundary.
\end{theorem}
\begin{proof} Assume that $(\cM,\cV,\conn)$ has negligible boundary. We will first show that $\Lap_{s}$ is symmetric.  For $w,\,z\in
\dom(\Lap_{s})$, we have
\[
\inprod{-\tr\conn[2](\conn[2]w), z} = \inprod{\conn[2]w,\conn[2]z}=\inprod{w,-\tr\conn[2](\conn[2]z)},
\]
where in the first equality we used~(\ref{E:n-b}) with $u=z$ and $v=\conn[2]w$, and in the second equality we used~(\ref{E:n-b}) with $u=w$ and $v=\conn[2]z$. This shows that $\Lap_{s}$ is symmetric. We now show that $\close{\Lap_{s}}=\Lap_{G}$. Taking closures in~(\ref{E:D-N-s}) and using Proposition~\ref{L:closures-N-D} we obtain
\begin{equation}\label{E:D-N-G}
\Lap_{D}\subset\close{\Lap_{s}}\quad\textrm{ and }\quad\Lap_{N}\subset\close{\Lap_{s}}.
\end{equation}
Since $\Lap_{s}$ is symmetric,  so is $\close{\Lap_{s}}$. Since $\Lap_{D}$ and $\Lap_{N}$ are self-adjoint and $\close{\Lap_{s}}$ is symmetric, from~(\ref{E:D-N-G}) we get $\Lap_{D}=\close{\Lap_{s}}=\Lap_{N}$. Therefore, $\close{\Lap_{s}}$ self-adjoint, that is, $\Lap_{s}$ is essentially self-adjoint.

Now assume that $\Lap_{s}$ is essentially self-adjoint. Taking closures in~(\ref{E:D-N-s}) and using Proposition~\ref{L:closures-N-D} we obtain~(\ref{E:D-N-G}), which leads to $\Lap_{D}=\close{\Lap_{s}}=\Lap_{N}$, that is, $\Sob[0]{1,2}(\cV)=\Sob{1,2}(\cV)$.  Now by
Theorem \ref{T:self-G} and Proposition \ref{Prop:lapgdn}, it follows that $(\cM,\cV,\conn)$ has negligible boundary.
\end{proof}

To summarise, we present the following list of
equivalences.
We note that this easily follows from Theorem~\ref{T:self-G},
\ref{T:esa-Lap-G-infty} and Proposition~\ref{Prop:lapgdn}.

\begin{corollary}
\label{Cor:TheEnd}
The following equivalences hold:
\begin{enumerate}[(i)]
\item the triplet $(\cM,\cV,\conn)$ has negligible boundary,
\item $\Sob[0]{1,2}(\cV)=\Sob{1,2}(\cV)$,
\item the operator $\Lap_{G}$ is self-adjoint,
\item the operator $\Lap_{s}$ is essentially self-adjoint.
\end{enumerate}
\end{corollary}

We conclude this paper with the following natural
questions that have risen out of our analysis. 

\begin{question}
\label{Q:1}
Does there exist a manifold $(\cM,\mg)$
and two vector bundles $(\cV_1,\mh_1,\conn_1)$
and $(\cV_2, \mh_2,\conn_2)$
with $\mh_i$ and $\conn_i$ compatible ($i = 1,2$), so that
$\Sob[0]{1,2}(\cV_1) = \Sob{1,2}(\cV_1)$
but $\Sob[0]{1,2}(\cV_2) \neq \Sob{1,2}(\cV_2)$?
\end{question}

\begin{question}
Does there exist a manifold $(\cM,\mg)$
and a vector bundle $(\cV, \mh, \conn)$
with $\conn$ and $\mh$ compatible
such that  either:
$\Sob[0]{1,2}(\cM) = \Sob{1,2}(\cM)$
	and $\Sob[0]{1,2}(\cV) \neq \Sob{1,2}(\cV)$, or,
$\Sob[0]{1,2}(\cV) = \Sob{1,2}(\cV)$
	and $\Sob[0]{1,2}(\cM) \neq \Sob{1,2}(\cM)$?
\end{question}

Thus, it is necessary that 
these questions be considered
only in the case that $(\cM,\mg)$
is geodesically incomplete. 
In this situation,
we do not expect the negligible
boundary property for one vector bundle
to necessarily follow from another.
Hence, by
Corollary \ref{Cor:TheEnd},
we at least expect Question
\ref{Q:1} to have an affirmative answer.

\begin{thebibliography}{10}

\bibitem{ADMc}
David Albrecht, Xuan Duong, and Alan McIntosh, \emph{Operator theory and
  harmonic analysis}, Instructional {W}orkshop on {A}nalysis and {G}eometry,
  {P}art {III} ({C}anberra, 1995), Proc. Centre Math. Appl. Austral. Nat.
  Univ., vol.~34, Austral. Nat. Univ., Canberra, 1996, pp.~77--136. \MR{1394696
  (97e:47001)}

\bibitem{Bandara-13}
Lashi Bandara, \emph{Density problems on vector bundles and manifolds}, Proc.
  Amer. Math. Soc. \textbf{142} (2014), no.~8, 2683--2695. \MR{3209324}

\bibitem{B-Rough}
Lashi {Bandara}, \emph{{Rough metrics on manifolds and quadratic estimates}},
  ArXiv e-prints (2014).

\bibitem{BMc}
Lashi Bandara and Alan McIntosh, \emph{The {K}ato square root problem on vector
  bundles with generalised bounded geometry}, The Journal of Geometric Analysis
  (2015), 1--35.

\bibitem{BMS}
M.~Braverman, O.~Milatovich, and M.~Shubin, \emph{Essential selfadjointness of
  {S}chr\"odinger-type operators on manifolds}, Uspekhi Mat. Nauk \textbf{57}
  (2002), no.~4(346), 3--58. \MR{1942115 (2004g:58021)}

\bibitem{Chernoff}
Paul~R. Chernoff, \emph{Essential self-adjointness of powers of generators of
  hyperbolic equations}, J. Functional Analysis \textbf{12} (1973), 401--414.
  \MR{0369890 (51 \#6119)}

\bibitem{Cordes}
H.~O. Cordes, \emph{Self-adjointness of powers of elliptic operators on
  non-compact manifolds}, Math. Ann. \textbf{195} (1972), 257--272. \MR{0292111
  (45 \#1198)}

\bibitem{E-88}
J{\"u}rgen Eichhorn, \emph{Elliptic differential operators on noncompact
  manifolds}, Seminar {A}nalysis of the {K}arl-{W}eierstrass-{I}nstitute of
  {M}athematics, 1986/87 ({B}erlin, 1986/87), Teubner-Texte Math., vol. 106,
  Teubner, Leipzig, 1988, pp.~4--169. \MR{1066741 (91k:58132)}

\bibitem{Gaffney}
Matthew~P. Gaffney, \emph{The harmonic operator for exterior differential
  forms}, Proc. Nat. Acad. Sci. U. S. A. \textbf{37} (1951), 48--50.
  \MR{0048138 (13,987b)}

\bibitem{GM}
Alexander Grigor'yan and Jun Masamune, \emph{Parabolicity and stochastic
  completeness of manifolds in terms of the {G}reen formula}, J. Math. Pures
  Appl. (9) \textbf{100} (2013), no.~5, 607--632. \MR{3115827}

\bibitem{Kato}
Tosio Kato, \emph{Perturbation theory for linear operators}, Classics in
  Mathematics, Springer-Verlag, Berlin, 1995, Reprint of the 1980 edition.
  \MR{1335452 (96a:47025)}

\bibitem{Le-Si}
Herbert Leinfelder and Christian~G. Simader, \emph{Schr\"odinger operators with
  singular magnetic vector potentials}, Math. Z. \textbf{176} (1981), no.~1,
  1--19. \MR{606167 (82d:35073)}

\bibitem{Masamune-99}
Jun Masamune, \emph{Essential self-adjointness of {L}aplacians on {R}iemannian
  manifolds with fractal boundary}, Comm. Partial Differential Equations
  \textbf{24} (1999), no.~3-4, 749--757. \MR{1683058 (2000m:58035)}

\bibitem{Masamune}
\bysame, \emph{Analysis of the {L}aplacian of an incomplete manifold with
  almost polar boundary}, Rend. Mat. Appl. (7) \textbf{25} (2005), no.~1,
  109--126. \MR{2142127 (2006a:58040)}

\bibitem{Masamune-05-CPDE}
\bysame, \emph{Essential self-adjointness of a sublaplacian via heat equation},
  Comm. Partial Differential Equations \textbf{30} (2005), no.~10-12,
  1595--1609. \MR{2182306 (2006h:35035)}

\bibitem{Petersen}
Peter Petersen, \emph{Riemannian geometry}, Graduate Texts in Mathematics, vol.
  171, Springer-Verlag, New York, 1998. \MR{1480173 (98m:53001)}

\bibitem{Shubin}
M.~A. Shubin, \emph{Pseudodifferential operators and spectral theory}, second
  ed., Springer-Verlag, Berlin, 2001, Translated from the 1978 Russian original
  by Stig I. Andersson. \MR{1852334 (2002d:47073)}

\bibitem{BC-2009}
E.P. van~den Ban and M.~Crainic, \emph{Analysis on manifolds: Lecture notes for
  the 2009/2010 master class}, University of Utrecht,
  http://www.staff.science.uu.nl/$\sim$crain101/AS-2013/main.pdf.

\end{thebibliography}
\providecommand{\bysame}{\leavevmode\hbox to3em{\hrulefill}\thinspace}
\providecommand{\MR}{\relax\ifhmode\unskip\space\fi MR }
\providecommand{\MRhref}[2]{%
  \href{http://www.ams.org/mathscinet-getitem?mr=#1}{#2}
}
\providecommand{\href}[2]{#2}

\setlength{\parskip}{0mm}

\end{document}